\documentclass[10pt,reqno]{amsproc}
\pdfoutput=1

\usepackage{mathpazo} %
\usepackage[scaled]{helvet} %
\usepackage{courier} %
\normalfont
\usepackage[T1]{fontenc}

\usepackage{hyperref}

\usepackage[margin=1.05in]{geometry}

\theoremstyle{theorem}
\newtheorem{theorem}{Theorem}

\newtheorem{lemma}[theorem]{Lemma}

\theoremstyle{definition}
\newtheorem*{remark}{Remark}

\DeclareMathOperator*{\GCD}{GCD}
\DeclareMathOperator{\ord}{ord}

\title{On the Greatest Common Divisor of Binomial Coefficients ${n \choose q}, {n \choose 2q}, {n \choose 3q},\dots$}
\author{Carl McTague}
\email{carl.mctague@rochester.edu}
\urladdr{\href{http://www.mctague.org/carl}{www.mctague.org/carl}}
\address{Mathematics Department, University of Rochester, Rochester, NY 14627, USA}

\begin{document}

\begin{abstract}
Every binomial coefficient aficionado$^1$ knows that the greatest common divisor of the binomial coefficients ${n \choose 1},{n \choose 2},\dots,{n \choose n-1}$ equals $p$ if $n=p^i$ for some $i>0$ and equals 1 otherwise. It is less well known that the greatest common divisor of the binomial coefficients ${2n \choose 2},{2n \choose 4},\dots,{2n\choose 2n-2}$ equals (a certain power of 2 times) the product of all odd primes $p$ such that $2n=p^i+p^j$ for some $i\le j$. This note gives a concise proof of a tidy generalization of these facts.
\end{abstract}

\maketitle

\let\thefootnote\relax\footnotetext{$^1$The author regards himself less \emph{aficionado} than \emph{espont\'aneo}, cf \cite[p.~52]{naked-lunch}.}

\begin{theorem}[{\cite{ram-1909}}]
  For any integer $n>1$:
  \begin{align*}
    \GCD_{0<k<n} {n \choose k} =
    \begin{cases}
      p & \text{if $n=p^i$ for some prime $p$ and some integer $i>0$} \\
      1 & \text{otherwise}
    \end{cases}
  \end{align*}
\end{theorem}

\begin{theorem} [{Lemma~12 of \cite{mctague-2014}}]
  For any integer $n>1$ and any prime $p>2$:
  \begin{align*}
    \ord_p \Big[ \GCD_{0<k<n} {2n \choose 2k} \Big] =
    \begin{cases}
      1 & \text{if $2n=p^i+p^j$ for some integers $0\le i\le j$} \\
      0 & \text{otherwise}
    \end{cases}
  \end{align*}
  where $\ord_p(m)$ is the highest power of $p$ dividing an integer $m$.
\end{theorem}

\begin{remark}
For a given integer $n>1$, at most one prime $p$ divides the $\GCD$ in Theorem~1. But more than one prime can divide the $\GCD$ in Theorem~2, which is why $\ord_p$ is used to state it. For example, if $n=3$ then $2n=3^1+3^1=5^0+5^1$ and indeed $\GCD_{0<k<3}{6\choose 2k}=15=3\cdot5$. In fact, more than two primes can divide: if $n=15$ then $2n=3^1+3^3=5^1+5^2=29^0+29^1$ and indeed $\GCD_{0<k<15}{30\choose 2k}=435=3\cdot5\cdot29$.
\end{remark}

These theorems are special cases of a (new) more general result:
\newtheorem*{thmq*}{Theorem Q}
\begin{thmq*}
  \label{mainthm}
  For any integers $n>q>0$, and for any prime $p$ congruent to $1$ modulo~$q$:
  \begin{align*}
    \ord_p \Big[ \GCD_{0<k<n/q} {n \choose qk} \Big] =
    \begin{cases}
      1 & \text{if $\alpha_p(n)\le q$} \\
      0 & \text{otherwise}
    \end{cases}
  \end{align*}
  where $\alpha_p(n)$ is the sum of the digits of the base-$p$ expansion of $n$, equivalently the smallest integer $r$ such that $n=p^{i_1}+\cdots+p^{i_r}$ for integers $0\le i_1\le\dots\le i_r$.
\end{thmq*}

\begin{remark}
  Since $p$ is congruent to $1$ modulo $q$, the inequality $\alpha_p(n)\le q$ is equivalent to the equality $\alpha_p(n)=s$, where $s$ is the unique integer in the range $0<s\le q$ congruent to $n$ modulo $q$. (Indeed, since $p$ is congruent to $1$ modulo $q$, so is each power~$p^i$, so $\alpha_p(n)$ is congruent to $n$ modulo~$q$.) For example, for $n>1$:
  \begin{align*}
    \ord_p \Big[ \GCD_{0<k<n} {qn \choose qk} \Big] =
    \begin{cases}
      1 & \text{if $\alpha_p(qn)=q$} \\
      0 & \text{otherwise}
    \end{cases}
  \end{align*}
  while:
  \begin{align*}
    \ord_p \Big[ \GCD_{0<k\le n} {qn+1 \choose qk} \Big] =
    \begin{cases}
      1 & \text{if $\alpha_p(qn+1)=1$} \\
      0 & \text{otherwise}
    \end{cases}
  \end{align*}
  When $q=2$, the former is Theorem~2, while the latter a priori extends Theorem~2. However, due to the symmetry of Pascal's triangle ${n \choose k}={n \choose n-k}$, this extension can already be deduced from Theorem~1.
\end{remark}

\begin{remark}
  The hypothesis that $p$ is congruent to $1$ modulo~$q$ was chosen for its balance of simplicity and generality, and is used in two different ways in the proof of Theorem~Q (below). It can be weakened, for example, to $p>q$ being relatively prime and $p^{i_1}\equiv\cdots\equiv p^{i_r}$ modulo $q$. (In the last paragraph of the proof, replace $(p-1)p^{i_r-1}$ with $qp^{i_r-1}$ when $p^{i_1}\equiv\cdots\equiv p^{i_r}\not\equiv1$ modulo $q$.) But it cannot be eliminated altogether since, for example, $\ord_2\big[\GCD_{0<k<2}{6\choose 3k}\big]=\ord_2(20)=2$.
\end{remark}

\begin{remark}
  A different generalization of Theorem~1 is obtained in \cite{oestreicher-steinig-1985} by determining the greatest common divisor of ${n\choose r}, {n\choose r+1},\dots,{n\choose s}$ for any $r\le s\le n$.
\end{remark}

The proof of Theorem~Q relies on:
\newtheorem*{kummerthm*}{Kummer's Theorem}
\begin{kummerthm*}[{\cite{kummer-1852}, cf \cite[\S1]{granville97}}]
  For any integers $0\le k\le n$ and any prime~$p$:
  \begin{align*}
    \ord_p\Big[{n \choose k}\Big] = \# \big\{ \text{carries when adding $k$ to $n-k$ in base $p$} \big\}
  \end{align*}
\end{kummerthm*}

In particular, it relies on the following consequence of Kummer's theorem:

\begin{lemma}
  \label{lemma:subseq}
  Given two integers $0\le k\le n$, write their base-$p$ expansions in the form:
  \begin{align*}
    k&=p^{j_1}+\cdots+p^{j_s} &
    n&=p^{i_1}+\cdots+p^{i_r}
  \end{align*}
  with $r$ and $s$ minimal, $i_1\le\cdots\le i_r$ and $j_1\le\cdots\le j_s$. Then $\ord_p[{n\choose k}]=0$ if and only if $(j_1,\dots,j_s)$ is a subsequence of $(i_1,\dots,i_r)$.
\end{lemma}

\begin{proof}[Proof of Lemma~\ref{lemma:subseq}]
  By Kummer's theorem, $\ord_p[{n\choose k}]=0$ if and only if there are no carries when adding $k$ to $n-k$ in base~$p$. This happens if and only if each base-$p$ digit of $k$ is $\le$ the corresponding base-$p$ digit of $n$. And this in turn is equivalent to $(j_1,\dots,j_s)$ being a subsequence of $(i_1,\dots,i_r)$.
\end{proof}

\begin{proof}[Proof of Theorem~Q]
To begin, note that for any set $S$ of integers:
\begin{align*}
  \ord_p[\GCD_{m\in S} m] = \min_{m\in S} \ord_p(m)
  \end{align*}
  So this order equals 0 if there is an integer $m$ in $S$ with $\ord_p(m)=0$. Similarly, this order equals 1 if (a) for every integer $m$ in $S$, $\ord_p(m)>0$  and (b) there is an integer $m$ in $S$ with $\ord_p(m)=1$.

  \medskip

  Now, write the base-$p$ expansion of $n$ in the form:
  \begin{align*}
    n=p^{i_1}+\cdots+p^{i_r}
  \end{align*}
  with $r$ minimal and $i_1\le\cdots\le i_r$.
  
  \medskip

  If $r>q$ then by Lemma~\ref{lemma:subseq}:
  \begin{align*}
    \ord_p\Big[{p^{i_1}+\cdots\cdots+p^{i_r} \choose p^{i_1}+\cdots+p^{i_q}}\Big]=0
  \end{align*}
  Since $p$ is congruent to $1$ modulo $q$, so is each power $p^i$, so $p^{i_1}+\cdots+p^{i_q}$ is divisible by $q$, and it follows that $\ord_p[\GCD_{0<k<n/q} {n \choose qk}]=0$.

  \medskip

  If $r\le q$ then $p^{j_1}+\cdots+p^{j_s}$ is not divisible by $q$ for any nonempty proper subsequence $(j_1,\dots,j_s)$ of $(i_1,\dots,i_q)$. Therefore, by Lemma~\ref{lemma:subseq}, $\ord_p{n\choose qk}>0$ for any $k$ with $0<qk<n$. So $\ord_p[\GCD_{0<k<n/q}{n\choose qk}]>0$.

  The largest exponent $i_r$ must be $>0$ since otherwise $n=p^0+\cdots+p^0=r\le q$, and by assumption $n>q$. Since $r$ is minimal, it equals the sum $\alpha_p(n)$ of the base-$p$ digits of $n$, so this sum is by assumption $\le q$. And $q<p$ since $p$ is prime and congruent to $1$ modulo $q$. It follows that the $(i_r-1)$st base-$p$ digit of $n$ is less than $p-1$. So there is exactly one carry when adding $(p-1)p^{i_r-1}$ to $n-(p-1)p^{i_r-1}$. By Kummer's theorem then:
  \begin{align*}
    \ord_p\Big[{p^{i_1}+\cdots+p^{i_r} \choose (p-1)p^{i_r-1}}\Big]=1
  \end{align*}
  Since $p$ is congruent to $1$ modulo $q$, $(p-1)p^{i_r-1}$ is divisible by $q$, and it follows that $\ord_p[\GCD_{0<k<n/q} {n\choose qk}]=1$.
\end{proof}

\noindent \textit{Thanks to Doug Ravenel, David Gepner and Marcus Zibrowius for helpful conversations. Thanks to the referee who suggested generalizing an earlier version of Theorem~Q. Thanks to G\"unter Ziegler for pointing out \cite{ram-1909}. Thanks to the villains \cite{binom-villainy-2014} who haunt the Hopkins mathematics department for helping inspire this work.}
\def\cprime{$'$}

\end{document}